\documentclass{amsart}
\usepackage{amsfonts}
\usepackage{amssymb}
\usepackage{amsmath}
\usepackage[dvipdf]{graphicx}
\newtheorem{theorem}{Theorem}

\newtheorem{corollary}{Corollary}

\newtheorem{proposition}{Proposition}
\newtheorem{remark}{Remark}

\title{Complete Foliations of Space Forms by Hypersurfaces}

\author{A. Caminha}
\address{Departamento de Matem\'atica, Universidade Federal do Cear\'a, Fortaleza,
Cear\'a, Brazil. 60455-760}
\email{antonio.caminha@gmail.com}

\author{P. Sousa}
\address{Departamento de Matem\'atica, Universidade Federal do Piau\'{\i}, Teresina,
Piau\'{\i}, Brazil. 64049-550}
\email{pauloalexandre.ufpi@gmail.com}

\author{F. Camargo}
\address{Departamento de Matem\'atica, Universidade Federal de Campina Grande, Campina Grande,
Paraíba, Brazil. 58109-970}
\email{feccamargo@yahoo.com.br}



\thanks{The first author is partially supported by CNPq}

\begin{document}
\maketitle

\begin{abstract}
We study foliations of space forms by complete hypersurfaces, under some mild conditions on its higher order mean curvatures. In particular, in Euclidean space we obtain a Bernstein-type theorem for graphs whose mean and scalar curvature do not change sign but may otherwise be nonconstant. We also establish the nonexistence of foliations of the standard sphere whose leaves are complete and have constant scalar curvature, thus extending a theorem of Barbosa, Kenmotsu and Oshikiri. For the more general case of {\em r-}minimal foliations of the Euclidean space, possibly with a singular set, we are able to invoke a theorem of Ferus to give conditions under which the nonsigular leaves are foliated by hyperplanes. 
\end{abstract}

\section{Introduction}

Codimension-one foliations of Riemannian spaces have been studied, through the geometric point of view, since the beginnings of the last century, when S. Bernstein~\cite{Bernstein:15}, proved that the only entire minimal graphs in $\mathbb R^3$ are planes. This result was later extended by J. Simons~\cite{Simons:68}, for entire minimal graphs in $\Bbb{R}^{n+1}$ up to $n=7$, and disproved by E. Bombieri, E. de Giorgi and E. Giusti~\cite{Bombieri:68} in all higher dimensions. We refer the reader to a paper of B. Nelli and M. Soret~\cite{Nelli:05} for a brief account of interesting related results on Bernstein's problem, as it became known these days.

A natural extension to the problem above is to consider codimension one complete foliations of space forms, whose leaves have constant mean curvature. In this respect, J. L. Barbosa, K. Kenmotsu and G. Oshikiri~\cite{Barbosa:91} proved that such a foliation must have minimal leaves if the ambient space is flat, and does not exist in the sphere.
Related results for graphs in products $M\times\mathbb R$ were also obtained by J. L. Barbosa, G. P. Bessa and J. F. Montenegro~\cite{Barbosa:08}, by imposing some restrictions on the fundamental tone of the Laplacian on the graph.

In this paper we study foliations of space forms by complete hypersurfaces, asking that the leaves have bounded second fundamental form and two consecutive higher order mean curvatures not changing signs. For the particular case of a graph in Euclidean space whose defining function satisfies certain growth conditions, in Theorem~\ref{thm:application to graphs} we are thus able to use a result of D. Ferus (Theorem 5.3 of~\cite{Dajczer:90}) to get a lower estimate on the relative nullity of the graph; we also discuss some examples that show that our hypotheses are not superfluous. As an interesting consequence, we obtain in Corollary~\ref{coro:application to graphs 3} a Bernstein-type theorem for such a graph, provided its mean and scalar curvature do not change sign (but may otherwise be nonconstant).

For the case of general, transversely orientable foliations of space forms, we follow the approach of~\cite{Barbosa:91}, computing in Proposition~\ref{prop:divergence on PrX on L and on the ambient} the divergence of the vector field $P_r\overline D_NN$ on a leaf of the foliation; here, $N$ is a unit vector field on the ambient space, normal to the leaves, and $P_r$ is the $r-$th Newton transformation of a leaf with respect to $N$. We are then able to extend one of the above mentioned theorems of~\cite{Barbosa:91}, proving the nonexistence of foliations of the standard sphere whose leaves are complete and have constant scalar curvature greater than one. We also consider a more direct generalization of the problem of Bernstein, i.e., that of the study of $r-$minimal foliations (possibly with a singular set) of the Euclidean space. In this setting, we are also able to rely to Ferus' theorem to prove that the nonsigular leaves are foliated by hyperplanes of a certain codimension, provided the $r-$th curvature of them does not vanish. We remark that problems of this kind have already been considered by the first author in the Lorentz setting~\cite{Caminha:06}. 

Besides the formula for the divergence of $P_r\overline D_NN$, another central tool for our work is a further elaboration, undertaken in Proposition~\ref{prop:first corollary of Yau 76} and Corollary~\ref{coro:Lr versio of corollary of Yau 76}, of S. T. Yau's extension (cf.~\cite{Yau:76}) of H. Hopf's theorem on subharmonic functions on complete noncompact Riemannian manifolds.

\section{Graphs in Euclidean space}

In what follows, unless otherwise stated, all spaces under consideration are supposed to be connected.

In the paper~\cite{Yau:76}, S. T. Yau obtained the following version of Stokes' theorem on an $n-$dimensional,
complete noncompact Riemannian manifold $M$: if $\omega\in\Omega^{n-1}(M)$, an
$n-1$ differential form on $M$, then there exists a
sequence $B_i$ of domains on $M$, such that $B_i\subset B_{i+1}$, $M=\cup_{i\geq 1}B_i$ and
$$\lim_{i\rightarrow+\infty}\int_{B_i}\omega=0.$$

By applying this result to $\omega=\iota_{\nabla f}$, where $f:M\rightarrow\mathbb R$ is a smooth function, $\nabla f$ denotes its gradient and $\iota_{\nabla f}$ the contraction in the direction of $\nabla f$, Yau established the following extension of H. Hopf's theorem on a complete noncompact Riemannian manifold: a subharmonic function whose gradient has integrable norm on $M$ must actually be harmonic.

We begin by extending the above result a little further. In what follows, we suppose $M$ oriented by the volume element $dM$, and let $\mathcal L^1(M)$ be the space of Lebesgue integrable functions on $M$. 

\begin{proposition}\label{prop:first corollary of Yau 76}
Let $X$ be a smooth vector field on the $n$ dimensional complete, noncompact, oriented Riemannian manifold $M^n$,
such that ${\rm div}X$ does not change sign on $M$. If $|X|\in\mathcal L^1(M)$, then ${\rm div}X=0$ on $M$.
\end{proposition}

\begin{proof}
Suppose, without loss of generality, that ${\rm div}X\geq 0$ on $M$. Let $\omega$ be the $(n-1)-$form in $M$ given by $\omega=\iota_XdM$, i.e., the contraction of $dM$ in the direction of a smooth vector field $X$ on $M$. If $\{e_1,\ldots,e_n\}$ is an orthonormal frame on an open set $U\subset M$,
with coframe $\{\omega_1,\ldots,\omega_n\}$, then
$$\iota_XdM=\sum_{i=1}^n(-1)^{i-1}\langle X,e_i\rangle\omega_1\wedge\ldots
\wedge\widehat\omega_i\wedge\ldots\wedge\omega_n.$$

Since the $(n-1)-$forms $\omega_1\wedge\ldots\wedge\widehat\omega_i\wedge\ldots
\wedge\omega_n$ are orthonormal in $\Omega^{n-1}(M)$, we get $$|\omega|^2=
\sum_{i=1}^n\langle X,e_i\rangle^2=|X|^2.$$
Then $|\omega|\in\mathcal L^1(M)$ and $d\omega=d(\iota_XdM)=({\rm div}X)dM$. Letting $B_i$ be as in the
preceeding discussion, we get
$$\int_{B_i}({\rm div}X)dM=\int_{B_i}d\omega\stackrel{i}{\longrightarrow}0.$$
But since ${\rm div}X\geq 0$ on $M$, it follows that ${\rm div}X=0$ on $M$.
\end{proof}

Now, let $\overline M^{n+1}$ be an $(n+1)$-dimensional Riemannian manifold. If $M$ is a complete, orientable, immersed hypersurface on $\overline M$, oriented by the choice of a smooth unit vector field $N$, we let $A:TM\to TM$ be the shape operator of $M$, i.e., $AX=-\overline D_XN$, where $\overline D$ stands for the Levi-Civitta connection of
$\overline M$. For $0\leq r\leq n$, the $r-$th Newton tensor $P_r$ on $M$ is recursively defined by
$$P_r=S_rI-AP_{r-1},$$
where $P_0=I$, the identity operator on each tangent space of $M$, and $S_r$ is the $r$-th elementary symmetric function
of the eigenvalues of $A$ (we also set $S_0=1$ and $S_r=0$ if $r>n$). A trivial induction shows that
\begin{equation}\label{eq:Pr as a polynomial in A}
P_r=\sum_{j=0}^r(-1)^jS_{r-j}A^{(j)},
\end{equation}
where $A^{(j)}$ denotes the composition of $A$ with itself, $j$ times ($A^{(0)}=I$).

One step ahead, let $f$ be a smooth function on $M$ and $L_rf={\rm tr}(P_r{\rm Hess}f)$. Then $L_0$ is the
Laplacian of $M$ and, if $\overline M$ has constant sectional curvature, H. Rosenberg proved in~\cite{Rosenberg:93} that
$L_rf={\rm div}(P_r\nabla f)$, where ${\rm div}$ stands for the divergence on $M$. Concerning this setting, one gets the following consequence of Proposition~\ref{prop:first corollary of Yau 76}.

\begin{corollary}\label{coro:Lr versio of corollary of Yau 76}
Let $x:M^n\rightarrow Q^{n+1}(a)$ be a complete oriented hypersurface of a space form $Q^{n+1}(a)$, with
bounded second fundamental form. If $f:M\rightarrow\mathbb R$ is a smooth function such that
$|\nabla f|\in\mathcal L^1(M)$ and $L_rf$ does not change sign on $M$, then $L_rf=0$ on $M$.
\end{corollary}

\begin{proof}
If $A$ is the second fundamental form of the immersion, then its eigenvalues are continuous functions on $M$. It thus
follows from (\ref{eq:Pr as a polynomial in A}) that $||P_r||$ is bounded on $M$ whenever $||A||$ is itself bounded on $M$. Therefore, there exists a constant $c>0$ such that $||P_r||\leq c$ on $M$, and hence
$$|P_r\nabla f|\leq||P_r||\,|\nabla f|\leq c|\nabla f|\in\mathcal L^1(M).$$
Since $L_rf={\rm div}(P_r\nabla f)$ does not change sign on $M$, proposition~\ref{prop:first corollary of Yau 76} gives $L_rf=0$ on $M$.
\end{proof}

We now specialize our discussion to the case of a complete oriented hypersurface $x:M^n\rightarrow\mathbb R^{n+1}$. If
$U$ is a parallel vector field in $\mathbb R^{n+1}$, we let $f,g:M\rightarrow\mathbb R$ be given by
\begin{equation}\label{eq:the functions f and g}
f=\langle N,U\rangle\ \ \text{and}\ \ g=\langle x,U\rangle,
\end{equation}
where, as before, $N$ is the unit normal vector field on $M$ that gives its orientation. Letting $U^{\top}$ denote the orthogonal projection of $U$ onto $M$, standard computations (cf.~\cite{Rosenberg:93})
give
\begin{equation}\label{eq:gradients of f and g}
\nabla f=U^{\top},\,\,\nabla g=-A(U^{\top}),
\end{equation}
\begin{equation}\label{eq:Lr of f}
L_rf=-(S_1S_{r+1}-(r+2)S_{r+2})f+U^{\top}(S_{r+1}),
\end{equation}
\begin{equation}\label{eq:Lr of g}
L_rg=-(r+1)S_{r+1}f.
\end{equation}

Specializing a little more, let $u:\mathbb R^n\rightarrow\mathbb R$ be a smooth function and $M^n\subset\mathbb R^{n+1}$
be the graph of $u$, i.e.,
$$M^n=\{(x_1,\ldots,x_n,u(x_1,\ldots,x_n))\in\mathbb R^{n+1};\,(x_1,\ldots,x_n)\in\mathbb R^n\}.$$
We also make $U=(-V,1)$ in the above discussion, where $V$ is a parallel vector field in $\mathbb R^n$.
Following R. Reilly~\cite{Reilly:73}, we can take $N=\frac{1}{W}(-{\rm grad}\,u,1)$ as a unit normal vector field
on $M$, where ${\rm grad}\,u$ is the gradient of $u$ on $\mathbb R^n$ and $W=\sqrt{1+|{\rm grad}\,u|^2}$. This way,
$$U^{\top}=U-\langle U,N\rangle N=\frac{1}{W^2}({\rm grad}\,u-V,\langle {\rm grad}\,u,{\rm grad}\,u-V\rangle),$$
so that $|U^{\top}|\leq\frac{1}{W}|{\rm grad}\,u-V|$. Therefore,
$$\int_M|U^{\top}|dM\leq\int_{\mathbb R^n}\frac{1}{W}|{\rm grad}\,u-V|Wdx=\int_{\mathbb R^n}|{\rm grad}\,u-V|dx,$$
and this is finite if, for instance, there exist positive constants $R$, $c$ and $\alpha$ such that
$|{\rm grad}\,u(p)-V|\leq\frac{c}{|p|^{n+\alpha}}$ whenever $|p|>R$.
We also point out that, in standard coordinates, the second fundamental form of $M$ with respect to the above choice of unit normal vector field is $\frac{1}{W}{\rm Hess}\,u$, where by ${\rm Hess}\,u$ we mean the Hessian form of $u$ on $\mathbb R^n$; hence, the condition that it is bounded amounts to the existence of a constant $c>0$ for which
$$||{\rm Hess}\,u||^2\leq c(1+|{\rm grad}\,u|^2).$$

We can now state and prove the following

\begin{theorem}\label{thm:application to graphs}
Let $M^n\subset\mathbb R^{n+1}$ be the graph of a smooth function $u:\mathbb R^n\rightarrow\mathbb R$, such that
$|\nabla u-V|\in\mathcal L^1(\mathbb R^n)$ for some $V\in\mathbb R^n$ and
$||{\rm Hess}\,u||^2\leq c(1+|{\rm grad}\,u|^2)$, for some $c>0$. If there exists $0\leq r\leq n-1$ such that the elementary symmetric functions $S_{r+1}$ and $S_{r+2}$ do not change sign on $M$, then $M$ has relative nullity $\nu\geq n-r$. In particular, if $S_r\neq 0$, then the graph is foliated by hyperplanes of dimension $n-r$.
\end{theorem}

\begin{proof}
Letting $f$ and $g$ be as in (\ref{eq:the functions f and g}), it follows from our hypotheses that both
$|\nabla f|$ and $|\nabla g|$ are integrable on $M$. On the other hand, since $M$ is a graph, the function $f$ is either
positive or negative on $M$. Since $S_{r+1}$ doesn't change sign on $M$, (\ref{eq:Lr of g}) assures that the same is true of
$L_rg$, and it follows from Corollary~\ref{coro:Lr versio of corollary of Yau 76} that $L_rg=0$ on $M$. In turn, this last information guarantees that $S_{r+1}$ vanishes on $M$, so that (\ref{eq:Lr of f}) gives
$$L_rf=(r+2)S_{r+2}f.$$
By applying the same reasoning (since $S_{r+2}$ also doesn't change sign on $M$), we get $L_rf=0$ on $M$, and hence
$S_{r+2}=0$ on $M$. Finally, since $S_{r+1}=S_{r+2}=0$, Proposition 1 of~\cite{Caminha:06} gives $S_j=0$ for all
$j\geq r+1$, so that $\nu\geq n-r$.

The last claim follows from a theorem of D. Ferus (theorem 5.3 of~\cite{Dajczer:90}).
\end{proof}

We now have immediately the following Bernstein-type result, where it is not assumed that the hypersurface has constant mean curvature.


\begin{corollary}\label{coro:application to graphs 3}
Let $M^n\subset\mathbb R^{n+1}$ be the graph of a smooth function $u:\mathbb R^n\rightarrow\mathbb R$, such that
$|\nabla u-V|\in\mathcal L^1(\mathbb R^n)$ for some $V\in\mathbb R^n$ and $||{\rm Hess}\,u||^2\leq c(1+|{\rm grad}\,u|^2)$, for some $c>0$. If the mean and scalar curvatures of $M$ do not change sign on it, then $M$ is the hyperplane on $\mathbb R^{n+1}$ orthogonal to $(-V,1)$.
\end{corollary}

\begin{proof}
Letting $H$ and $R$ respectively denote the mean and scalar curvatures of $M$, just note that $S_1=nH$ and (by Gauss' equation) $n(n-1)R=2S_2$, so that $S_1$ and $S_2$ do not change sign on $M$. By the previous result, $M$ has relative nullity $n$ and, since it is complete, it is a hyperplane. The rest follows from our previous discussions.
\end{proof}

\begin{remark}
To see that the conditions on $u$ are not superfluous, consider the following two examples:
\begin{enumerate}
\item If $u(x_1,\ldots,x_n)=(x_1^2+\cdots+x_r^2)(\alpha_{r+1}x_{r+1}+\cdots+\alpha_nx_n)$,
where $\alpha_{r+1},\ldots,\alpha_n$ are real constants, not all zero. If $M$ is the graph of $u$, then, out of the
hyperplane $\alpha_{r+1}x_{r+1}+\cdots+\alpha_nx_n=0$, $M$ has index of relative nullity exactly equal to $n-r$;
in particular, $S_{r+1}=S_{r+2}=0$. On the other hand, $|\nabla u-V|\notin\mathcal L^1(\mathbb R^n)$ for any
$V\in\mathbb R^n$ and there is no $c>0$ such that $||{\rm Hess}\,u||^2\leq c(1+|{\rm grad}\,u|^2)$ for all
$x\in\mathbb R^n$.
\item If $u(x_1,\ldots,x_n)=x_1^2+\cdots+x_n^2$ and $M$ is the graph of $u$, then $S_1,S_2>0$ on $M$ and
$||{\rm Hess}\,u||^2\leq 4n(1+|{\rm grad}\,u|^2)$, although $|\nabla u-V|\notin\mathcal L^1(\mathbb R^n)$ for any
$V\in\mathbb R^n$.
\end{enumerate}
\end{remark}

\section{Foliations of space forms}

We now turn our attention to a more general situation, namely, we consider codimension one foliations of
Riemannian manifolds and try to understand the effect of higher curvatures on the leaves. We remark that, for foliations whose leaves have constant mean curvature, this problem has been considered by Barbosa, Kenmotsu and Oshikiri in~\cite{Barbosa:91}, and also by Bessa, Barbosa and Montenegro in~\cite{Barbosa:08}. 

As before, $\overline M^{n+1}$ is an $(n+1)$-dimensional orientable Riemannian manifold and $\mathcal{F}$ a
smooth foliation of codimension one in $\overline M$. Recall (cf.~\cite{LinsNeto:84}) that $\mathcal{F}$ is transversely orientable if we can choose a smooth unit vector field $N$, defined on $\overline M$, that is
normal to the leaves of $\mathcal{F}$. If this is the case, then, for each $p\in\overline M$, we consider the linear
operator $A:T_{p}\overline M\to T_{p}\overline M$ defined by $A(Y(p))=-\overline D_{Y(p)}N$, where, as before,
$\overline D$ denotes the Levi-Civitta connection of $\overline M$. It is clear that if $Y$ is a smooth vector field on $\overline M$,
then the same is true of $A(Y)$. Moreover, letting $A_{L}$ denote the second fundamental form of a leaf $L$ of
$\mathcal F$, we get $A_{|L}=A_{L}$. Accordingly, we let $P_{r}:T_{p}\overline M\to T_{p}\overline M$ be the
linear operator that coincides with the $r-$th Newton transformation on each leaf of the foliation.

Following~\cite{Barbosa:91}, we let $X=\overline D_NN$, so that $X$ is tangent to the leaves of the foliation and independent of the the choice of the field $N$. In what follows, we compute the divergence of $P_{r}(X)$ on $\overline M$ and on a leaf $L$ of $\mathcal F$.

\begin{proposition}\label{prop:divergence on PrX on L and on the ambient}
Let $\mathcal{F}$ be a smooth, transversely orientable foliation of codimension one of a Riemannian manifold
$\overline M^{n+1}$, $N$ a unit vector field on $\overline M$, normal to the leaves of $\mathcal{F}$ and $X=\overline{D}_{N}N$.
If $L$ is a leaf of $\mathcal F$, then
\begin{eqnarray}\label{eq:div(PrX)}
{\rm div}_{L}(P_{r}(X))&=&\sum_{i=1}^{n}\langle\overline{R}(N,e_{i})N,P_{r} (e_{i})
\rangle+\langle X,{\rm div}_{L}P_{r}\rangle\\
&&+{\rm tr}(A^2P_r)+\langle X,P_r(X)\rangle-N(S_{r+1}),\nonumber
\end{eqnarray}
where $\overline R$ is the curvature tensor of $\overline M$, $\{e_i\}$ is an orthormal frame on $L$ and ${\rm tr}(\,\cdot\,)$ stands for the trace in $L$ for the operator in parentheses. Moreover,
\begin{equation}\label{eq:divergence of PrX on the ambient}
{\rm div}_{\overline M}P_{r}(X)={\rm div}_{L}P_{r}(X)-\langle P_{r}(X),X\rangle.
\end{equation}
\end{proposition}

\begin{proof}
Given a point $p\in L$, choose an adapted frame field $\{e_1,\ldots,e_n,e_{n+1}\}$ defined in a
neighborhood of $p$ in $\overline M$, i.e., an orthonormal set of vector fields such that $e_{1},...,e_{n}$
are tangent to the leaves and $e_{n+1}=N$. Ask further that $A(e_{i}(p))=\lambda_{i}e_{i}(p)$, for all $1\leq i\leq n$. If we call $D$ the Levi-Civitta connection of $L$ (and, as before, $\overline D$ that of $\overline M$), then
\begin{eqnarray*}
{\rm div}_{L}P_{r}(X)&=&\sum_{i=1}^{n}\langle D_{e_{i}}P_{r}(X),e_{i}\rangle
=\sum_{i=1}^{n}e_{i}\langle P_{r}(X),e_{i}\rangle-\sum_{i=1}^{n}\langle P_{r}(X),D_{e_i}e_i\rangle\\
&=&\sum_{i=1}^{n}e_{i}\langle X,P_{r}(e_{i})\rangle-\sum_{i=1}^{n}\langle X,P_{r}(D_{e_{i}}e_{i})\rangle\\
&=&\sum_{i=1}^{n}e_{i}\langle\overline{D}_{N}N,P_{r}(e_{i})\rangle-\sum_{i=1}^{n}\langle
\overline{D}_{N}N,P_r(D_{e_i}e_{i})\rangle\\
&=&\sum_{i=1}^{n}\langle\overline{D}_{e_{i}}\overline{D}_{N}N,P_{r}(e_{i})\rangle+
\sum_{i=1}^{n}\langle\overline{D}_{N}N,D_{e_i}P_r(e_i)\rangle\\
&&-\sum_{i=1}^n\langle\overline D_NN,P_r(D_{e_i}e_i)\rangle\\
&=&\sum_{i=1}^{n}\langle\overline{R}(N,e_{i})N,P_{r}(e_{i})\rangle+\sum_{i=1}^{n}\langle
\overline{D}_{N}\overline{D}_{e_{i}}N,P_{r}(e_{i})\rangle\\
&&-\sum_{i=1}^{n}\langle\overline{D}_{[N,e_{i}]}N,P_{r}(e_{i})\rangle
+\sum_{i=1}^{n}\langle\overline{D}_{N}N,D_{e_i}P_r(e_i)\rangle\\
&&-\sum_{i=1}^{n}\langle\overline{D}_{N}N,P_r(D_{e_i}e_i)\rangle\\
&=&\sum_{i=1}^{n}\langle\overline{R}(N,e_{i})N,P_{r}(e_{i})\rangle
-\sum_{i=1}^{n}\langle\overline{D}_{N}A(e_{i}),P_{r}(e_{i})\rangle\\
&&-\sum_{i=1}^{n}\langle\overline{D}_{[N,e_{i}]}N,P_{r}(e_{i})\rangle+\sum_{i=1}^{n}\langle
\overline{D}_{N}N,D_{e_{i}}P_{r}(e_{i})-P_{r}(D_{e_{i}}e_{i})\rangle.
\end{eqnarray*}

Now, substituting the equality
$$[N,e_{i}]=\sum_{j=1}^{n}\langle[N,e_{i}],e_{j}\rangle e_{j}+\langle[N,e_{i}],N\rangle N$$
into the above, we get

\begin{eqnarray*}
{\rm div}_{L}P_{r}(X)&=&\sum_{i=1}^{n}\langle\overline{R}(N,e_{i})N,P_{r}(e_{i})\rangle
-N\Big(\sum_{i=1}^{n}\langle A(e_{i}),P_{r}(e_{i})\rangle\Big)\\
&&+\sum_{i=1}^{n}\langle A(e_{i}),\overline{D}_{N}P_{r}(e_{i})\rangle
-\sum_{i,j=1}^{n}\langle[N,e_{i}],e_{j}\rangle\langle\overline{D}_{e_{j}}N,P_{r}(e_{i})\rangle\\
&&-\sum_{i=1}^{n}\langle[N,e_{i}],N\rangle\langle\overline{D}_{N}N,P_{r}(e_{i})\rangle
+\langle X,{\rm div}_{L}P_{r}\rangle\\
&=&\sum_{i=1}^{n}\langle\overline{R}(N,e_{i})N,P_{r}(e_{i})\rangle
-N\Big(\sum_{i=1}^{n}\langle e_{i},AP_{r}(e_{i})\rangle\Big)\\
&&+\sum_{i=1}^{n}\langle A(e_{i}),\overline{D}_{N}P_{r}(e_{i})\rangle+\langle X,{\rm div}_{L}P_{r}\rangle\\
&&-\sum_{i,j=1}^{n}\langle\overline{D}_{e_{i}}N,e_{j}\rangle\langle A(e_{j}),P_{r}(e_{i})\rangle
+\sum_{i,j=1}^{n}\langle\overline{D}_{N}e_{i},e_{j}\rangle\langle A(e_{j}),P_{r}(e_{i})\rangle\\
&&+\stackrel{=0}{\overbrace{\sum_{i=1}^{n}\langle\overline{D}_{e_{i}}N,N\rangle\langle X,P_{r}
(e_{i})\rangle}}-\sum_{i=1}^{n}\langle\overline{D}_{N}e_{i},N\rangle\langle X,P_{r}(e_{i})\rangle\\
&=&\sum_{i=1}^{n}\langle\overline{R}(N,e_{i})N,P_{r}(e_{i})\rangle-N({\rm tr}AP_{r})+\langle X,{\rm div}_{L}
P_{r}\rangle\\
&&+\sum_{i=1}^{n}\langle A(e_{i}),\overline{D}_{N}P_{r}(e_{i})\rangle+
\sum_{i,j=1}^{n}\langle A(e_{i}),e_{j}\rangle\langle A(e_{j}),P_{r}(e_{i})\rangle\\
&&+\sum_{i,j=1}^{n}\langle\overline{D}_{N}e_{i},e_{j}\rangle\langle A(e_{j}),P_{r}(e_{i})\rangle
+\sum_{i=1}^{n}\langle e_{i},\overline{D}_{N}N\rangle\langle X,P_{r}(e_{i})\rangle\\
&=&\sum_{i=1}^{n}\langle\overline{R}(N,e_{i})N,P_{r}(e_{i})\rangle-N({\rm tr}AP_{r})+\langle X,
{\rm div}_{L}P_{r}\rangle\\
&&+\sum_{i=1}^{n}\langle A(e_{i}),\overline{D}_{N}P_{r}(e_{i})\rangle+\sum_{i,j=1}^{n}\langle A(e_{i}),
e_{j}\rangle\langle e_{j},AP_{r}(e_{i})\rangle\\
&&+\sum_{i,j=1}^{n}\langle\overline{D}_{N}e_{i},e_{j}\rangle\langle A(e_{j}),P_{r}(e_{i})\rangle
+\sum_{i=1}^{n}\langle e_{i},\overline{D}_{N}N\rangle\langle P_{r}(X),e_{i}\rangle\\
&=&\sum_{i=1}^{n}\langle\overline{R}(N,e_{i})N,P_{r}(e_{i})\rangle-N({\rm tr}AP_{r})+\langle X,
{\rm div}_{L}P_{r}\rangle\\
&&+\sum_{i=1}^{n}\langle A(e_{i}),\overline{D}_{N}P_{r}(e_{i})\rangle+\sum_{i=1}^{n}\langle A(e_{i}),
AP_{r}(e_{i})\rangle\\
&&+\sum_{i,j=1}^{n}\langle\overline{D}_{N}e_{i},e_{j}\rangle\langle A(e_{j}),P_{r}(e_{i})\rangle
+\langle\overline{D}_{N}N,P_{r}(X)\rangle\\
&=&\sum_{i=1}^{n}\langle\overline{R}(N,e_{i})N,P_{r}(e_{i})\rangle-N({\rm tr}AP_{r})+\langle X,
{\rm div}_{L}P_{r}\rangle\\
&&+{\rm tr}A^{2}P_{r}+\langle X,P_{r}(X)\rangle
+\sum_{i=1}^{n}\langle A(e_{i}),\overline{D}_{N}P_{r}(e_{i})\rangle\\
&&+\sum_{i,j=1}^{n}\langle\overline{D}_{N}e_{i},e_{j}\rangle\langle A(e_{j}),P_{r}(e_{i})\rangle.
\end{eqnarray*}

In order to understand the last two summands above, let $l_{ij}=\langle\overline{D}_{N}e_{i},e_{j}\rangle$ and
$m_{ji}=\langle A(e_{j}),P_{r}(e_{i})\rangle$. It is not difficult to verify that
$l_{ij}=-l_{ji}$ and $m_{ij}=m_{ji}$, so that
$$\sum_{i,j=1}^{n}\langle\overline{D}_{N}e_{i},e_{j}\rangle\langle A(e_{j}),P_{r}(e_{i})\rangle=\sum_{i,j=1}^{n}l_{ij}m_{ji}=0.$$
On the other hand,
\begin{eqnarray*}
\sum_{i=1}^{n}\langle A(e_{i}),\overline{D}_{N}P_{r}(e_{i})\rangle&=&\sum_{i,j=1}^{n}\langle A(e_{i}),
e_{j}\rangle\langle\overline{D}_{N}P_{r}(e_{i}),e_{j}\rangle\\
&=&\sum_{i,j=1}^{n}\langle A(e_{i}),e_{j}\rangle N\Big(\langle P_{r}(e_{i}),e_{j}\rangle\Big)\\
&&-\sum_{i,j=1}^{n}\langle A(e_{i}),e_{j}\rangle\langle P_{r}(e_{i}),\overline{D}_{N}e_{j}\rangle\\
&=&\sum_{i,j=1}^{n}\langle A(e_{i}),e_{j}\rangle N\Big(\langle P_{r}(e_{i}),e_{j}\rangle\Big)\\
&&-\sum_{i,j,k=1}^{n}\langle A(e_{i}),e_{j}\rangle\langle
P_{r}(e_{i}),e_{k}\rangle\langle e_{k},
\overline{D}_{N}e_{j}\rangle.
\end{eqnarray*}
Letting $h_{ij}=\langle A(e_{i}),e_{j}\rangle$ and $t_{ik}=\langle P_{r}(e_{i}),e_{k}\rangle$,
we get $h_{ij}=h_{ji}$ and $t_{ik}=t_{ki}$, and hence
$$\sum_{i,j,k=1}^{n}\langle A(e_{i}),e_{j}\rangle\langle P_{r}(e_{i}),e_{k}\rangle\langle e_{k},
\overline{D}_{N}e_{j}\rangle=\sum_{i,j,k=1}^{n}h_{ij}t_{ik}l_{jk}=0.$$
Therefore,
\begin{eqnarray*}
\sum_{i=1}^{n}\langle A(e_{i}),\overline{D}_{N}P_{r}(e_{i})\rangle&=&\sum_{i,j=1}^{n}\langle A(e_{i}),e_{j}
\rangle N\Big(\langle P_{r}(e_{i}),e_{j}\rangle\Big)=\sum_{i,j=1}^{n}h_{ij}N(t_{ij})\\
&=&N\Big(\sum_{i,j=1}^{n}h_{ij}t_{ij}\Big)-\sum_{i,j=1}^{n}N(h_{ij})t_{ij}\\
&=&N({\rm tr}(AP_r))-\sum_{i,j=1}^{n}N(h_{ij})t_{ij}.
\end{eqnarray*}
Now, by means of computations analogous to those leading to (17), on page 193 of~\cite{Caminha:06}, we conclude
that $\sum_{i,j=1}^{n}N(h_{ij})t_{ij}=N(S_{r+1})$ at $p$, and this concludes the proof of (\ref{eq:div(PrX)}).

It is now an easy matter to get (\ref{eq:divergence of PrX on the ambient}):
\begin{eqnarray*}
{\rm div}_{\overline M}P_{r}(X)&=&\sum_{i=1}^{n}\langle\overline{D}_{e_{i}}P_{r}(X),e_{i}\rangle
+\langle\overline{D}_{N}P_{r}(X),N\rangle\\
&=&\sum_{i=1}^{n}\langle\overline{D}_{e_{i}}P_{r}(X),e_{i}\rangle-\langle P_{r}(X),
\overline{D}_{N}N\rangle\\
&=&{\rm div}_{L}P_{r}(X)-\langle P_{r}(X),X\rangle.
\end{eqnarray*}
\end{proof}

\begin{remark}\label{rem:caso de ambiente com curvatura constante}
 Concerning the above computations, if $\overline M^{n+1}$ has constant sectional curvature, then Rosenberg proved in~\cite{Rosenberg:93} that ${\rm div}_LP_r=0$, thus simplifying {\rm(\ref{eq:div(PrX)})}. We shall use this fact twice in what follows.
\end{remark}

We now study codimension-one foliations of $\Bbb{S}^{n+1}$ whose leaves have constant scalar curvature, thus extending Corollary $3.5$ of~\cite{Barbosa:91}\footnote{As is the case of~\cite{Barbosa:91} (since even-dimensional spheres cannot have transversely orientable foliations), the interesting case is that of odd-dimensional spheres. However, since the proof does not distinguish between odd and even, we present it in general form.}. 

\begin{theorem}
There is no smooth, transversely orientable foliation of codimension one of the Euclidean sphere $\Bbb{S}^{n+1}$, whose
leaves are complete and have constant scalar curvature greater than one.
\end{theorem}

\begin{proof}
Suppose there exists a foliation $\mathcal{F}$ of $\Bbb{S}^{n+1}$ with the properties above, let $N$ be a unit vector field on $\mathbb S^{n+1}$ normal to the leaves and $A_L(\,\cdot\,)=-\overline D_{(\,\cdot\,)}N$ be the shape operator of a leaf $L$ with respect to $N$. If $R_{L}$ denotes the constant value of the scalar curvature of the leaf $L$ of $\mathcal{F}$, it follows from Gauss' equation that $2S_{2}=n(n-1)(R_{L}-1)$, so that $S_{2}$ is a positive constant.

If $\lambda_1,\ldots,\lambda_n$ are the eigenvalues of $A_L$, then
$$S_{1}^{2}=|A|^{2}+2S_{2}>|A|^{2}\geq\lambda_{i}^{2}.$$
Choosing the orientation in such a way that $S_{1}>0$, it follows from the above inequalities that
$S_{1}-\lambda_{i}>0$. This says that $P_{1}$ is positive definite on $L$.

Since the scalar curvature function $R:\Bbb{S}^{n+1}\to\Bbb{R}$, that associates to each point
the value of the scalar curvature of the leaf of $\mathcal{F}$ through that point, is constant
on the leaves, Proposition 2.31 of~\cite{Barbosa:91} gives that either $R$ is constant on
$\Bbb{S}^{n+1}$, or there exists a compact leaf $L$ of $\mathcal{F}$ having the property that
$$R_{L}=\max_{p\in\Bbb{S}^{n+1}}R(p).$$

Assume first that $R$ is nonconstant on $\Bbb{S}^{n+1}$, and let $L$ be the compact leaf of $\mathcal F$ with maximal
scalar curvature, so that $N(S_{2})=0$ along $L$. The curvature operator of the sphere, together with Remark~\ref{rem:caso de ambiente com curvatura constante} and (\ref{eq:div(PrX)}), now give
$${\rm div}_{L}P_{1}(X)={\rm tr}(P_1)+{\rm tr}(A^{2}P_1)+\langle X,P_{1}(X)\rangle>0.$$
On the other hand, since $L$ is compact, divergence theorem applied to $L$ gives
${\rm div}_{L}P_{1}(X)=0$, which is a contradiction.

Now, assume that $R$ is constant on $\Bbb{S}^{n+1}$. Then $N(S_{2})=0$, and (\ref{eq:div(PrX)}) and
(\ref{eq:divergence of PrX on the ambient}) give
$${\rm div}P_{1}(X)={\rm tr}(P_1)+{\rm tr}(A^{2}P_1)>0.$$
However, integration over $\Bbb{S}^{n+1}$ yields ${\rm tr}(P_1)={\rm tr}(A^{2}P_1)=0$, which contradics the positive definiteness of $P_1$. This concludes the proof of the theorem.
\end{proof}

\begin{remark}
We point out that there are several families of compact tori in $\mathbb S^{n+1}$ with constant scalar curvature greater than one, and refer the reader to Example $4.4$ of~\cite{Caminha2:06} for the details. Of course, none of them constitutes a foliation of $\mathbb S^{n+1}$.
\end{remark}

We finish this paper with a generalization of Theorem~\ref{thm:application to graphs} to a singular foliation of
$\mathbb R^{n+1}$, by which we mean a foliation $\mathcal F$ of $\mathbb R^{n+1}\setminus S$, where $S\subset\mathbb R^{n+1}$ is a set of Lebesgue measure zero. In order to state the result, if $\mathcal{F}$ is a transversely orientable such foliation of $\Bbb{R}^{n+1}$, with unit normal vector field $N$ normal to the leaves, then (as before) we let $X=\overline D_NN$, where $\overline D$ is the Levi-Civitta connection of $\mathbb R^{n+1}$. We also recall the reader that an isometric immersion $x:M^n\rightarrow\overline M^{n+1}$ is said to be $r-$minimal if $S_{r+1}=0$ on $M$.

\begin{theorem}
Let $\mathcal{F}$ be a smooth, transversely orientable singular foliation of codimension one of $\Bbb{R}^{n+1}$, whose
leaves are complete, $r-$minimal and such that $S_r$ doesn't change sign on them. If $|X|\in\mathcal L^{1}$
and $|A|$ is bounded along each leaf, then the relative nullity of each leaf is at least $n-r$. In particular, if
$S_r\neq 0$ on a leaf, then this leaf is foliated by hyperplanes of dimension $n-r$.
\end{theorem}

\begin{proof}
Let $L$ be a leaf of $\mathcal F$. Since $S_r$ doesn't change sign on $L$, we again have $P_{r}$ semi-definite by a result of J. Hounie and M. L. Leite~\cite{Hounie:99}, so that ${\rm tr}(A^2P_r)$ and $\langle X,P_r(X)\rangle$ are both nonnegative or both nonpositive on $L$. Therefore, by applying (\ref{eq:div(PrX)}) and Remark~\ref{rem:caso de ambiente com curvatura constante} again, we get
$${\rm div}_{L}(P_{r}(X))={\rm tr}(A^2P_r)+\langle X,P_r(X)\rangle,$$
which is either greater than or less than zero on $L$. It thus follows from
Proposition~\ref{prop:first corollary of Yau 76} that ${\rm div}_{L}P_{r}(X)=0$, and, since $S_{r+1}=0$ on $L$, we get
$${\rm tr}(A^{2}P_{r})=-(r+2)S_{r+2}=0.$$

This way, as before we get $S_k=0$ for all $k\geq r+1$, and it suffices to reason as in the end of the proof of
Theorem~\ref{thm:application to graphs}, invoking Ferus' theorem.
\end{proof}

\begin{remark}
As an example of the situation described in the theorem above, one has the singular foliation of $\mathbb R^{n+1}$ by
the concentric cylinders $\mathbb S_R^r\times\mathbb R^{n-r}$. Here, $\mathbb S_R^r\subset\mathbb R^{r+1}$ denotes the
sphere with center $0\in\mathbb R^r$ and radius $R>0$; the singular set of the foliation is the $(n-r)-$hyperplane
$\{0\}\times\mathbb R^{n-r}$ in $\mathbb R^{n+1}$.
\end{remark}

\end{document}